\documentclass[a4paper,10pt]{iopart}

\usepackage{amsmath}
\usepackage{amsfonts}
\usepackage{amsthm}
\usepackage{amssymb}
\usepackage[mathscr]{eucal}
\usepackage{tikz}
\usetikzlibrary{arrows,matrix}
\usepackage{fp}
\usepackage{textcomp} %fuer \textmu
\usepackage{paralist}
\usepackage{cite}

\newtheoremstyle{thm}% name
  {\baselineskip}%      Space above
  {\baselineskip}%      Space below
  {\itshape}%           Body font
  {}%                   Indent amount (empty = no indent, \parindent = para indent)
  {\bf}%                Thm head font
  {.}%                  Punctuation after thm head
  {.5em}%               Space after thm head
  {}%                   Thm head spec
  
\newtheoremstyle{others}% name
  {\baselineskip}%      Space above
  {\baselineskip}%      Space below
  {\upshape}%           Body font
  {}%                   Indent amount (empty = no indent, \parindent = para indent)
  {\bf}%                Thm head font
  {.}%                  Punctuation after thm head
  {.5em}%               Space after thm head
  {}%                   Thm head spec

\theoremstyle{thm}
\newtheorem{theorem}{Theorem}[section]
\newtheorem{proposition}[theorem]{Proposition}

\newtheorem{corollary}[theorem]{Corollary}

\theoremstyle{others}
\newtheorem{definition}[theorem]{Definition}

\newtheorem{remark}[theorem]{Remark}

\DeclareMathOperator*{\argmin}{arg\,min}

\newlength{\arraycolseptmp}
\newcommand{\bigO}{\mathcal{O}}

\newcommand{\CI}{\operatorname{COR}_I}
\newcommand{\CIC}{\operatorname{COR}_{I^\complement}}
\newcommand{\convbi}{\ast}%{\circledast}

\newcommand{\hilbert}{\mathscr{H}}
\newcommand{\Id}{\operatorname{Id}}
\newcommand{\ip}[2]{\langle #1, #2 \rangle}
\newcommand{\Ip}[2]{\big\langle #1, #2 \big\rangle}

\newcommand{\kr}{\operatorname{ker}}
\newcommand{\majo}{\mathcal{M}}
\newcommand{\N}{\mathbb{N}}
\newcommand{\Nearrow}{\rotatebox[origin=c]{45}{$\Rightarrow$}}
\newcommand{\norm}[1]{\|#1\|}

\newcommand{\nsr}{r_{\eta / u}}
\newcommand{\R}{\mathbb{R}}

\newcommand{\rg}{\operatorname{rg}}
\newcommand{\Searrow}{\rotatebox[origin=c]{-45}{$\Rightarrow$}}
\newcommand{\set}[2]{\{#1\, | \,#2\}}
\newcommand{\Set}[2]{\big\{#1\, \big| \,#2\big\}}
\newcommand{\Sign}{\operatorname{Sign}}
\newcommand{\sign}{\operatorname{sign}}

\newcommand{\spa}{\operatorname{span}}
\newcommand{\sparse}{^{\scriptscriptstyle\lozenge}}
\newcommand{\sumstack}[2]{\renewcommand{\arraystretch}{0.5}
\setlength{\arraycolseptmp}{\arraycolsep}
\setlength{\arraycolsep}{0pt}
  \begin{array}{c}
    \scriptstyle #1 \\
    \scriptstyle #2
  \end{array}
  \setlength{\arraycolsep}{\arraycolseptmp}
  \renewcommand{\arraystretch}{1.0}}
\newcommand{\supp}{\operatorname{supp}}
\newcommand{\Z}{\mathbb{Z}}

\begin{document}

\title[Exact recovery with Tikhonov regularization]{Beyond convergence rates: Exact recovery with Tikhonov regularization with sparsity constraints}

\author{DA~Lorenz$^1$, S~Schiffler$^2$ and D~Trede$^{2,}$\footnote[7]{Author to whom correspondence shall be addressed.}}
\address{$^1$TU Braunschweig, Institute for Analysis and Algebra, Pockelsstr. 14, D-38118 Braunschweig, Germany\\
  $^2$Zentrum f\"ur Technomathematik, Universit\"at Bremen, Fachbereich Mathematik/Informatik, Postfach 33 04 40, D-28334 Bremen, Germany\\}

\eads{\mailto{d.lorenz@tu-braunschweig.de}, \mailto{schiffi@math.uni-bremen.de} and \mailto{trede@math.uni-bremen.de}}

\begin{abstract}
  The Tikhonov regularization of linear ill-posed problems with an
  $\ell^1$ penalty is considered. We recall results for linear
  convergence rates and results on exact
  recovery of the support.  Moreover, we derive conditions for exact
  support recovery which are especially applicable in the case of
  ill-posed problems, where other conditions, e.g.~based on the
  so-called coherence or the restricted isometry property are usually
  not applicable. The obtained results also show that the regularized
  solutions do not only converge in the $\ell^1$-norm but also in the
  vector space $\ell^0$ (when considered as the strict inductive limit
  of the spaces $\R^n$ as $n$ tends to infinity). Additionally, the
  relations between different conditions for exact support recovery
  and linear convergence rates are investigated.

  With an imaging example from digital holography the applicability of
  the obtained results is illustrated, i.e.~that one may check a
  priori if the experimental setup guarantees exact recovery with
  Tikhonov regularization with sparsity constraints.
\end{abstract}

%65J20 Numerical analysis - Numerical analysis in abstract spaces -  Improperly posed problems; regularization 
%47A52 Operator theory - General theory of linear operators - Ill-posed problems, regularization

\section{Introduction}
In this paper we consider linear inverse problems with a bounded linear operator
$A:\hilbert_1 \to \hilbert_2$ between two separable Hilbert spaces $\hilbert_1$ and $\hilbert_2$,
\begin{equation}
  Af = g.
  \label{eq_Tikhonov-l1-IP}
\end{equation}
We are given a noisy observation $g^\varepsilon = g + \eta \in
\hilbert_2$ with noise level $\norm{g-g^\varepsilon} 
\leq\varepsilon$ and try to reconstruct the solution $f$ of $Af = g$
from the knowledge of $g^\varepsilon$. We are especially interested in
the case in which~\eqref{eq_Tikhonov-l1-IP} is ill-posed in the sense
of Nashed, i.e.~when the range of $A$ is not closed. In particular
this implies that the (generalized) solution
of~\eqref{eq_Tikhonov-l1-IP} is unstable, or in other words, that the
generalized inverse $A^\dagger$ is unbounded. In this context,
regularization has to be employed to stably solve the
problem~\cite{engl1996inverseproblems}.

We assume that the operator equation $Af=g$ has a solution $f\sparse$ 
that can be expressed sparsely in an orthonormal basis $\Psi := \{\psi_i\}_{i\in \Z}$ of $\hilbert_1$,
i.e.~$f\sparse$ decomposes into a finite number of basis elements,
\[
  f\sparse = \sum\limits_{i\in\Z} u_i\sparse  \psi_i
  \quad \mbox{with} \quad
  u\sparse \in\ell^2(\Z,\R),\quad \mbox{and}\quad \big|\set{i\in\Z}{u_i\sparse \neq 0}\big|<\infty.%\norm{u\sparse }_{\ell^0} <\infty.
\]
The knowledge that $f\sparse$ can be expressed sparsely can be utilized
for the reconstruction by using an \emph{$\ell^1$-penalized Tikhonov} regularization~\cite{daubechies2003iteratethresh},
i.e.~an approximate solution is given as a minimizer of the
functional
\begin{equation}
  \tfrac{1}{2}
  \norm{Af - g^\varepsilon}^2_{\hilbert_2} + \alpha \sum\limits_{i\in\Z} |\ip{f}{\psi_i}|,
  \label{eq_Tikhonov-l1_HH}
\end{equation}
with regularization parameter $\alpha>0$.
In contrast to the classical Tikhonov functional with a quadratic penalty~\cite{engl1996inverseproblems},
the $\ell^1$-penalized functional promotes sparsity since small coefficients
are penalized more.

For the sake of notational simplification, we use $\ell^2 = \ell^2(\Z,\R)$ and introduce the synthesis operator
$D:\ell^2 \to \hilbert_1$, which for $u\in\ell^2$ is defined by $D u = \sum u_i \psi_i$.
With that and the definition 
$K:=A\circ D:\ell^2\to \hilbert_2$
we can rewrite the inverse problem~(\ref{eq_Tikhonov-l1-IP}) as $K u = g$.
Adopting the usual convention in convex analysis we use the following
somewhat sloppy notation
\[
\norm{\,\cdot\,}_{\ell^1}: \ell^2 \to [0,\infty],\quad \norm{u}_{\ell^1} =
\begin{cases}
  \norm{u}_{\ell^1}, & \text{if } u\in\ell^1,\\
  \infty, & \text{if } u\in\ell^2\setminus\ell^1,
\end{cases}
\]
and we rewrite the $\ell^1$-penalized Tikhonov
regularization~(\ref{eq_Tikhonov-l1_HH}) as
\begin{equation}\label{eq_Tikhonov_ell1}
  T_\alpha(u) := 
  \tfrac{1}{2}
  \norm{Ku - g^\varepsilon}^2_{\hilbert_2} + \alpha \norm{u}_{\ell^1}.
\end{equation}
In the following we frequently use the standard basis
of $\ell^2$, which is denoted by $\{e_j\}_{j\in\Z}$. 
The Tikhonov functional~(\ref{eq_Tikhonov_ell1}) has also been used in
the context of sparse recovery under the name Basis Pursuit
Denoising~\cite{Chen1998basispursuit}.

Daubechies et
al.~\cite{daubechies2003iteratethresh} showed that the minimization
of~\eqref{eq_Tikhonov_ell1} is indeed a regularization and derived
error estimates in a particular wavelet setting. Error estimates and
convergence rates under different source conditions have been derived
by Lorenz~\cite{lorenz2008reglp} and Grasmair et
al.~\cite{grasmair2008sparseregularization}. In this paper we aim at
conditions that ensure that the minimizers $u^{\alpha,\varepsilon} \in
\argmin T_\alpha(u)$ have the same support as $u\sparse$. In the
context of sparse recovery, this phenomenon is called \emph{exact recovery}. 
One of the main applications in the field of sparse recovery is compressive sampling,
a new sampling technique which allows to sample sparse signals at low
rates~\cite{Candes2006b}.
Our approach builds heavily on techniques and results from the field
of sparse recovery
from~\cite{Fuchs2004,fuchs2005exactsparse,Tropp2006relax,Dossal2005,Gribonval2008}
some of which we transfer to the field of inverse and ill-posed
problems. Very roughly spoken, the conditions for exact recovery can
be divided into two classes: Sharp conditions which are 
not practical since they rely on unknown quantities (in this
category are for example Tropp's ERC~\cite[theorem~8]{Tropp2006relax} and the
null space property~\cite{gribonval2007sparsenessmeasure}). More loose
conditions which are far from being necessary but seem more practical
(in this category are for example conditions using
incoherence~\cite{donoho2003sparse},~\cite[corollary~9]{Tropp2006relax} and the restricted isometry
property~\cite{candes2005decoding}). Moreover, the latter conditions
are usually not applicable for inverse and ill-posed problems (somehow
due to arbitrarily small singular values of the operator). Hence, we
focus on ``intermediate'' conditions~\cite{Dossal2005,Denis2009c} and
show how they can be applied to general inverse and ill-posed
problems. Especially we contribute the following point: While it is
good to know that exact recovery is possible for some regularization
parameter, what one really needs is a computable recipe to choose a
parameter which uses only available information and hence, we
especially treat this question on the choice of a regularization
parameter which guarantees exact recovery.

% The field of
% sparse recovery and ill-posed linear operator equations differ in
% several points. First, sparse recovery is formulated in finite
% dimensional space while ill-posed linear operator equations are
% formulated in infinite dimensional Banach and Hilbert spaces. As a
% consequence, different topologies can be considered (besides the weak
% and the strong topology of $\ell^2$, also the topologies of $\ell^1$
% and $\ell^\infty$ are meaningful in our case). Moreover, sparse
% recovery usually assumes an overcomplete dictionary for modeling the
% signal. The quality of the dictionary is expressed in terms of
% incoherence. The emphasis is on highly overcomplete dictionaries with
% small coherence. In ill-posed linear operator equations redundancy is
% often not needed but the problem comes from ill-posedness in the sense
% of arbitrarily small singular values. Unless the basis $\{\psi_i\}_{i\in\Z}$ is close
% to the singular vectors of the operator $A$ (i.e.~$\{A \psi_i\}_{i\in\Z}$ resp.~$\{K e_i\}_{i\in\Z}$ are nearly orthogonal), 
% this makes coherence concepts
% in general not applicable and one has to resort to other
% conditions, see e.g.~\cite{Dossal2005,Denis2009c}.

The paper is organized as follows.  In section
\ref{sec_ell1_preliminaries} we summarize some properties of
$\ell^1$-penalized Tikhonov minimizers and recall a stability result
from~\cite{grasmair2008sparseregularization}.  In section
\ref{sec:known-results-from} we review previous results on exact
recovery in the context of sparse recovery and we illustrate our
contribution.  Section~\ref{sec_ell1_ERC} contains the main
theoretical results of the paper. Especially known results on exact
recovery conditions are transfered to ill-posed problems and we give a
parameter choice rule which ensure exact recovery in the presence of
noise (under appropriate assumptions). One novelty here is, that we
focus on conditions and a choice rule which can be verified \emph{a
  priori} and hence are of practical relevance (and not only of
theoretical relevance). In section \ref{sec:relat-betw-recov} we
investigate the relation between the ERC from~\cite{Tropp2006relax},
the source condition and the null space
property~\cite{gribonval2007sparsenessmeasure}. In section
\ref{sec_applications}, we demonstrate the practicability of the
deduced recovery condition with an example from imaging, namely, an
example from digital holography.  In section \ref{sec-ell1_conclusion}
we give a conclusion on exact recovery conditions for Tikhonov
regularization with sparsity constraints.

\section{The $\ell^1$-penalized Tikhonov functional}\label{sec_ell1_preliminaries}
Before we start with error estimates, we recall some basic properties
of the $\ell^1$-penalized Tikhonov functional $T_\alpha$. First we
repeat a trivial characterization of the minimizer.

\begin{proposition}[Optimality condition]\label{optimality_condition_ell1}
  Define the \emph{set-valued sign function}
  $\Sign:\ell^2 \to \big\{\{-1\},[-1,+1],\{+1\}\big\}^\Z$,
  for $u\in\ell^2$, by
  \begin{equation}\label{eq-l1-set-valued-sign}
    \big(\Sign(u)\big)_k :=
    \begin{cases}
      \{-1\}, & u_k<0,\\
      [-1,+1], & u_k=0,\\
      \{+1\}, & u_k>0.
    \end{cases}
  \end{equation}
  Let $u^{\alpha,\varepsilon} \in \ell^2$.
  Then the following statements are equivalent:
  \begin{align}
    \text{(i)}  \qquad & u^{\alpha,\varepsilon} \in \argmin\limits_{u\in\ell^2} T_\alpha(u). \qquad\qquad \label{eq_l1-op-cond_0} \\
    \text{(ii)} \qquad & -K^*(K u^{\alpha,\varepsilon} - g^\varepsilon) \in \alpha \Sign (u^{\alpha,\varepsilon}). \label{eq_l1-op-cond}
  \end{align}
\end{proposition}
\begin{proof}
  Since the the set-valued sign function is actually the subgradient
  of the $\ell^1$ norm, the proof consists of noting that (ii) is just
  the optimality condition $0\in\partial
  T_\alpha(u^{\alpha,\varepsilon})$. Due to convexity of $T_\alpha$
  (ii) is also sufficient.
\end{proof}
Another well known characterization of a minimizer $u^{\alpha,\varepsilon}$ is,
that it is a fixed point of
$u^{\alpha,\varepsilon} = \mathbb{S}_\alpha (u^{\alpha,\varepsilon} + K^* (g^\varepsilon - K u^{\alpha,\varepsilon}))$,
where $\mathbb{S}_\alpha$ denotes the soft-thresholding operator,
cf. e.g.~\cite{daubechies2003iteratethresh}.
From this characterization or from~\eqref{eq_l1-op-cond} we can deduce the following:
Since the range of $K^*$ is contained in $\ell^2$,
any minimizer $u^{\alpha,\varepsilon}$ of the 
$\ell^1$-penalized Tikhonov functional $T_\alpha$
is finitely supported for every $\alpha>0$. (Note that this
observation relies on the fact that $K$ is bounded on $\ell^2$. If we
model $K:\ell^1\to\hilbert_2$ boundedly, as appropriate for normalized
dictionaries, we cannot conclude that $u^{\alpha,\varepsilon}$ is
finitely supported, since the adjoint operator maps $K^*:\hilbert_2 \to \ell^\infty$.)

Uniqueness of the minimizer of~(\ref{eq_Tikhonov_ell1}) could be guaranteed by ensuring strict convexity. This holds, e.g., if $K$ is injective. 
A weaker property of the operator $K$, which also guarantees uniqueness 
(although the functional is not strictly convex)
is the FBI~\cite{Bredies2008itersoftconvlinear} property defined below.

\begin{definition}
  Let $K:\ell^2 \to \hilbert_2$ be an operator mapping into a Hilbert space $\hilbert_2$. Then $K$ has the
  \emph{finite basis injectivity (FBI)} property, if for all finite subsets $J\subset\Z$
  the operator restricted to $\spa\set{e_i}{i\in J}$ is injective, i.e.~for all $u,v\in\ell^2$ with $Ku=Kv$ and $u_k=v_k=0$, for all $k\notin J$,
  it follows that $u=v$.
\end{definition}

In inverse problems with sparsity constraints the FBI property
is used for a couple of issues concerning $\ell^p$-penalized Tikhonov functionals,
for example for deduction of stability results~\cite{lorenz2008reglp,grasmair2008sparseregularization,Grasmair_suf_nec_ell1},
for derivation of efficient minimization schemes~\cite{Jin2009e,Griesse2008ssnsparsity},
and for proving convergence of minimization algorithms~\cite{Dahlke2009,Ramlau2011,Bredies2008itersoftconvlinear}.
A demon\-strative example for an operator which possesses the FBI property
but is not fully injective is the following:
Denote with $\{e_i\}_{i=0,\dots}$ the usual real Fourier basis of
$L^2([0,1])$ and with $\{\psi_j\}_{j=1,\dots}$ the Haar wavelet basis
of $L^2([0,1])$ and define the operator $K:\ell^2(\Z)\to L^2([0,1])$
by $Ku = \sum_{j=0}^\infty u_j e_j + \sum_{j=-1}^{-\infty}
u_j\psi_{-j}$. Then $K$ is clearly not injective, since any Haar
wavelet can be expressed in the Fourier basis and vice versa. However,
$K$ obeys the FBI property since neither any Haar wavelet is a finite
linear combination of elements of the Fourier bases nor the other way
round.

The FBI property is related to the so-called \emph{restricted isometry property (RIP)}~\cite{candes2005decoding} of a matrix,
which is a quite common assumption in the theory of compressive sampling~\cite{Candes2007,Candes2006b}.
The RIP is defined as follows.
Let $A$ be a $m\times n$ matrix and let $s<n$ be an integer.
The \emph{restricted isometry constant} of order $s$ is defined 
as the smallest number $0<c_s<1$, such that
the following condition
holds for all $v\in\R^n$ with at most $s$ non-zero entries:
\[
  (1-c_s) \norm{v}_{\ell^2}^2
  \leq \norm{A v}_{\ell^2}^2
  \leq (1+c_s) \norm{v}_{\ell^2}^2.
\]
Essentially, this property denotes that
the matrix is approximately an isometry when restricted to small
subspaces.  The FBI property, however, is defined for operators acting
on the sequence space and only says, that the restriction to finite
dimensional subspaces is still injective and makes no assumption of
the involved constants.

With $\ell^0$ we denote the vector space of all real-valued sequence
with only finitely many non-zero entries. In contrast to the $\ell^p$
spaces with $p>0$ there is no obvious (quasi-)norm available which
turns $\ell^0$ into a (quasi-)Banach space. We will come back to the
issue of defining a suitable topology on $\ell^0$ later.  In general,
the minimum-$\norm{\cdot}_{\ell^1}$ solution $u\sparse$ of $Ku=g$
neither needs to be in $\ell^0$, nor needs to be unique.  If we
assume that there is a finitely supported solution $u\sparse\in\ell^0$
of $Ku=g$, then the set of all solutions of $Ku=g$ is given by
$u\sparse + \ker K$.
If $K$ possesses the FBI property, then the solution $u\sparse$ is the
unique solution in $\ell^0$, hence $\ker K \subset
\ell^2\setminus\ell^0$.  However, in general $u\sparse\in\ell^0$ is
not a minimum-$\norm{\cdot}_{\ell^1}$ solution.  In the following we
assume that $K$ possesses the FBI property and denote the unique
solution of $Ku=g$ in $\ell^0$ with $u\sparse$.

Stability and convergence rates results for $\ell^1$-penalized
Tikhonov functionals have been deduced
in~\cite{daubechies2003iteratethresh,lorenz2008reglp,grasmair2008sparseregularization,Grasmair_suf_nec_ell1}.
The following error estimate
from~\cite{grasmair2008sparseregularization} ensures the linear
convergence to the minimum-$\norm{\cdot}_{\ell^1}$ solution,
if a certain source condition is satisfied.
We state it here in full detail and give explicit constants.

\begin{theorem}[{Error estimate~\cite[theorem 15]{grasmair2008sparseregularization}}]\label{theorem_ell1-convergence-rate}
  Let $K$ possess the FBI property, $u\sparse\in\ell^0$ with $\supp
  u\sparse = I$ be a minimum-$\norm{\cdot}_{\ell^1}$ solution of
  $Ku=g$, and $\norm{g-g^\varepsilon}_{\hilbert_2}\leq\varepsilon$.
  Let the following source condition (SC) be fulfilled:
  \begin{equation}
    \text{there exists $w\in\hilbert_2$ such that}\ K^*w= \xi \in\Sign(u\sparse).
    \label{eq_l1-soucecond-convergence rate}
  \end{equation}
  Moreover, let
  \[
  \theta = \sup \Set{|\xi_k|}{|\xi_k|<1}
  \]
  and $c>0$ such that for all $u\in\ell^2$ with $\supp (u) \subset I$ it holds
  \[
  \norm{Ku}\geq c\norm{u}.
  \]
  Then for the minimizers $u^{\alpha,\varepsilon}$ of $T_\alpha$
  it holds
  \begin{equation}
    \label{eq_l1-error_estimate}
    \norm{u^{\alpha,\varepsilon} - u\sparse}_{\ell^1} \leq \frac{\norm{K}+1}{1-\theta}\frac{\varepsilon^2}{\alpha} + \Bigl(\frac{1}{c} + \norm{w}\frac{\norm{K}+1}{1-\theta}\Bigr)(\alpha + \varepsilon).
  \end{equation}
  Especially, with $\alpha\asymp\varepsilon$ it holds
  \begin{equation}\label{eq_l1-linear_convergence}
    \norm{u^{\alpha,\varepsilon} - u\sparse}_{\ell^1}
    = \bigO(\varepsilon).
  \end{equation}
\end{theorem}

\begin{remark}
  \upshape
  \begin{enumerate}[a)]
  \item Since $\xi\in\ell^2$ by definition of $K$, it is clear that $\theta<1$.
  \item Since $K$ possesses the FBI property, the existence of $c>0$ is ensured.
  \item To achieve the linear convergence
    rate~(\ref{eq_l1-linear_convergence}), the SC~(\ref{eq_l1-soucecond-convergence rate}) is even
    necessary, cf.~\cite{Grasmair_suf_nec_ell1}.
  \end{enumerate}
\end{remark}

The above theorem is remarkable since it gives an error estimate for
regularization with a sparsity constraints with comparably weak
conditions of the operator, especially nothing is assumed about the
incoherence of $K$ in either way. However, the constants in the error
estimate~\eqref{eq_l1-error_estimate} are both depending on the
unknown quantities $I$, $\norm{w}$ and $\theta$ and are possibly huge
(especially $c$ can be small and $\theta$ can be close to one).

\section{Known results from sparse recovery}
\label{sec:known-results-from}
% The results presented in the next section build upon results
% from~\cite{Fuchs2004,fuchs2005exactsparse} and~\cite{Tropp2006relax}.

% In~\cite{Fuchs2004}, Fuchs gives a condition, which ensures that the
% support of the Tikhonov minimizer with noiseless data $g$, and the
% support of $u\sparse\in\ell^0$ coincide and
% in~\cite{fuchs2005exactsparse}, Fuchs transfers his results
% from~\cite{Fuchs2004} to noisy signals.  Assuming that the so-called
% coherence parameter $\mu:=\sup_{i\neq j}|\ip{K e_i}{K e_j}|$ is small,
% he proves a condition which ensures $\supp (u^{\alpha,\varepsilon}) =
% \supp (u\sparse)$.  
% Unless the dictionary is somehow close to the
% singular vectors of the operator,
% for ill-posed inverse problems the coherence
% parameter $\mu$ is typically huge (i.e.~$\mu \approx 1$) and hence
% Fuchs' results cannot be used.

In~\cite{Tropp2006relax} Tropp deduces a condition which ensures 
exact recovery.
To formulate the statement, we need the following notations.
For a subset $J\subset \Z$, we denote
with $P_J:\ell^2\to\ell^2$ the projection onto $\spa\set{e_i}{i\in J}$,
\[
  P_J u := \textstyle\sum_{j\in J} u_j e_j,
\]
i.e.~the coefficients $j\notin J$ are set to 0 and hence $\supp (P_J u) \subset J$.
With that definition $K P_J : \ell^2 \to \hilbert_2$ modifies the operator $K$ such that $KP_J u$ only depends on the entries $u_i$ for $i\in J$ and hence is something similar to the restriction of K
to  $\spa\set{e_i}{i\in J}$.

Moreover, for a linear operator $B$
we denote the pseudoinverse operator by $B^\dagger$.
With these definitions we are able to formulate Tropp's condition for
exact recovery. 

\begin{theorem}\label{theorem_tropp}
  Let $K$ be bounded and assume that $KP_I$ is injective; let
  $g^\varepsilon_I$ be the orthogonal projection of $g^\varepsilon$ to
  the range of $K P_I$ and denote with $u^\varepsilon_I \in \ell^2$
  the unique element with $\supp (u^\varepsilon_I) \subset I$
  such that $g^\varepsilon_I = K u^\varepsilon_I$.
  
  If the \emph{exact recovery condition} (ERC)
  \begin{equation}
    \label{eq_ERC}
    \sup_{i\in I^\complement} \norm{(K P_I)^\dagger K e_i}_{\ell^1} < 1,
  \end{equation}
  holds, then the parameter choice rule
  \begin{equation}
    \label{eq:para_rule_tropp}
    \frac{\sup_{i\in I^\complement} |\ip{g^\varepsilon - g^\varepsilon_I}{Ke_i}|}
    {1 - \sup_{i\in I^\complement} \norm{(K P_I)^\dagger K e_i}_{\ell^1}}
    < \alpha < 
    \frac{\min\limits_{i\in I} |u^\varepsilon_I (i)|}
    {\norm{(P_I K^* K P_I)^{-1}}_{\ell^1,\ell^1}}
  \end{equation}
  ensures that $\supp (u^{\alpha,\varepsilon}) = \supp (u\sparse)$.
\end{theorem}

This theorem can be extracted from~\cite{Tropp2006relax} under
slightly different assumptions (basically it is theorem~8 there,
however, this relies on several other results
in~\cite{Tropp2006relax}).

The applicability of this result is limited due to several terms: The
expressions $\norm{(P_I K^* K P_I)^{-1}}_{\ell^1,\ell^1}$ and
$\norm{(K P_I)^\dagger K e_i}_{\ell^1}$, need the knowledge of $I$
which is unknown. Moreover, the quantities $g^\varepsilon_I$ (the
projection of $g^\varepsilon$ onto the range of $KP_I$) and
$u^\varepsilon_I$ (the unique element with $\supp (u^\varepsilon_I) \subset I$ 
such that $g^\varepsilon_I = K u^\varepsilon_I$)
are unknown and not computable
without the knowledge of $I$.

In the following section we deduce a-priori parameter
rules that are easier to use than Tropp's parameter choice
rule~(\ref{eq:para_rule_tropp}). The idea
is to get rid of the expressions $g^\varepsilon_I$ and
$u^\varepsilon_I$, which cannot be estimated a priori.  Furthermore,
we will apply the techniques from~\cite{Gribonval2008,Dossal2005} to
deal with the terms $\norm{(P_I K^* K P_I)^{-1}}_{\ell^1,\ell^1}$ and
$\norm{(K P_I)^\dagger K e_i}_{\ell^1}$.
 
Finally we remark that there are conditions for exact recovery (also
in the presence of noise) which use the so called \emph{coherence} of
the dictionary in~\cite{Fuchs2004,fuchs2005exactsparse} and~\cite[corollary~9]{Tropp2006relax}. The
conditions are much easier to check (since they only rely on inner
products $\ip{Ke_i}{Ke_j}$) but are also much harder to fulfill in
practice.

\section{Beyond convergence rates: exact recovery for ill-posed
  operator equations}\label{sec_ell1_ERC}
In this paragraph we give an a priori parameter rule which ensures
that the unknown support of the sparse solution $u\sparse\in\ell^0$ is
recovered exactly, i.e. $\supp (u^{\alpha,\varepsilon}) = \supp
(u\sparse)$.  We assume that $K$ possesses the FBI property, and hence
$u\sparse$ is the unique solution of $Ku=g$ in $\ell^0$.  With $I$ we
denote the support of $u\sparse$, i.e.
\[
  I := \supp (u\sparse) := \set{i\in\Z}{u_i\sparse\neq 0}.
\]

\begin{theorem}[Lower bound on $\alpha$]\label{theorem_l1-ERC_noise1}
  Let $u\sparse \in \ell^0$, $\supp (u\sparse) = I$, and
  $g^\varepsilon = K u\sparse + \eta$ the noisy data.
  Assume %that the operator norm of $K$ is bounded by 1 and
  that $K$ is bounded and possesses the FBI property.
  If the following condition holds,
  \begin{equation}\label{eq_l1-ERC_noise}
    \sup\limits_{i\in I^\complement} \norm{(K P_I)^\dagger K e_i}_{\ell^1}
    < 1,
  \end{equation}
  then the parameter rule
  \begin{equation}\label{eq_l1-parameter_rule1}
    \alpha
    > \frac{1 + \sup_{i \in I^\complement} \norm{(KP_I)^\dagger Ke_i}_{\ell^1}}
           {1 - \sup_{i\in I^\complement} \norm{(K P_I)^\dagger K e_i}_{\ell^1}}
    \; \sup\limits_{i\in\Z} |\ip{\eta}{Ke_i}|
  \end{equation}
  ensures that the support of $u^{\alpha,\varepsilon}$ is contained in $I$.
\end{theorem}

\begin{proof}
  In~\cite[theorem 8]{Tropp2006relax}
  it is shown that condition~(\ref{eq_l1-ERC_noise}) together with the parameter rule
  \[
    \alpha >
    \frac{\sup_{i\in I^\complement} |\ip{g^\varepsilon - g^\varepsilon_I}{Ke_i}|}
         {1 - \sup_{i\in I^\complement} \norm{(K P_I)^\dagger K e_i}_{\ell^1}}
  \]
  ensures that the support of $u^{\alpha,\varepsilon} = \argmin T_\alpha(u)$ is contained in $I$.
  Recall that $g^\varepsilon_I$ is the orthogonal projection of $g^\varepsilon$
  to the range of $K P_I$, i.e. $g^\varepsilon_I = K P_I (K P_I)^\dagger g^\varepsilon$.
  Since $\Id - K P_I (K P_I)^\dagger$ equals the orthogonal projection on
  $\rg(K P_I)^\perp$ and $g = K u\sparse \in \rg(K P_I)$, we get
  \[
    g^\varepsilon - g^\varepsilon_I
    = \big(\Id - K P_I (K P_I)^\dagger \big) (g^\varepsilon - g)
    = \big(\Id - K P_I (K P_I)^\dagger \big) (\eta).
  \]
  Hence, using standard identities for the pseudo inverse and
  H\"ol\-der's inequality, we get that for $i\in I^\complement$ it
  holds that
  \begin{align*}
    |\ip{g^\varepsilon - g^\varepsilon_I}{Ke_i}|
    & = |\ip{(\Id - K P_I (K P_I)^\dagger) \, \eta}{Ke_i}|\\
    & = |\ip{\eta}{(\Id - KP_I (KP_I)^\dagger) Ke_i}|\\
    & = |\ip{K^*\eta}{(\Id - (KP_I)^\dagger K) e_i}_{\ell^2}|\\
    & \leq \norm{K^*\eta}_{\ell^\infty} \;
      \norm{(\Id - (KP_I)^\dagger K) e_i}_{\ell^1}.
  \end{align*}
  Since $i\in I^\complement$
  and $\supp((KP_I)^\dagger K e_i) \subset I$ we get
  \[
    \sup\limits_{i \in I^\complement} \norm{e_i - (KP_I)^\dagger Ke_i}_{\ell^1}
    = 1 + \sup\limits_{i \in I^\complement} \norm{(KP_I)^\dagger Ke_i}_{\ell^1}.
  \]
  Hence, using $\norm{K^*\eta}_{\ell^\infty} = \sup_{i\in\Z}
  |\ip{K^*\eta}{e_i}|$ we end in the estimate
  \begin{equation}
    \sup_{i\in I^\complement} |\ip{g^\varepsilon - g^\varepsilon_I}{Ke_i}|
    \leq \sup\limits_{i\in\Z} |\ip{\eta}{Ke_i}| \,
    \big(1 + \sup\limits_{i \in I^\complement} \norm{(KP_I)^\dagger Ke_i}_{\ell^1}\big).
    \label{eq_l1-ERC_noise_estimate1}
  \end{equation}
  Thus, the condition~(\ref{eq_l1-ERC_noise}) together with the parameter choice rule~(\ref{eq_l1-parameter_rule1}) ensures
  that the support is contained in $I$.
  A direct proof without using~\cite{Tropp2006relax}
  can be found in~\cite{Trede2010}.
\end{proof}

% \begin{remark}
%   \upshape
%   The assumption $\norm{K}\leq 1$
%   has been introduced for the sake of notational simplification.
%   However, afterwards we will need it for some estimates.
%   The assumption that $K$ is an operator with norm bounded by 1
%   is a common condition for Tikhonov functionals
%   with sparsity constraints, cf.~e.g.~\cite{daubechies2003iteratethresh}.
%   Actually, this is not really a limitation, because every
%   inverse problem $K u = g$ with bounded operator $K$ can be 
%   rescaled so that $\norm{K}\leq 1$ holds.
% \end{remark}

\begin{remark}\label{remark_l1-noise-level}
  \upshape
  Instead of using the estimate~(\ref{eq_l1-ERC_noise_estimate1}), one can
  alternatively use
  another more common upper bound for $\sup_{i\in I^\complement} |\ip{g^\varepsilon - g^\varepsilon_I}{Ke_i}|$. %$|\ip{(\Id - P_{\rg(K P_I)}) \, \eta}{Ke_i}|$.
  For that notice that $\Id - K P_I (K P_I)^\dagger$ 
  is the orthogonal projection on $\rg(K P_I)^\perp$.
  Then, since the norm of orthogonal projections is bounded by 1, we can estimate 
  for $i\in I^\complement$
  with the Cauchy-Schwarz inequality as follows
    \begin{equation}
      |\ip{\eta}{(\Id - K P_I (K P_I)^\dagger) Ke_i}| 
      \leq \norm{\eta}_{\hilbert_2} \norm{(\Id - K P_I (K P_I)^\dagger) Ke_i}_{\hilbert_2}
      \leq \varepsilon\norm{K}.
      \label{eq_l1-ERC_noise_estimate2}
    \end{equation}
  In general, one cannot say which estimate gives a sharper bound, inequality~(\ref{eq_l1-ERC_noise_estimate1}) or inequality~(\ref{eq_l1-ERC_noise_estimate2}).
  However, in practice the noise $\eta$ often is a realization of some random variable e.g.~with symmetric distribution
  and hence, $\sup_{i\in\Z} |\ip{\eta}{Ke_i}|\ll\varepsilon\norm{K}$ seems plausible.
  In this case the estimate with H\"older's inequality~(\ref{eq_l1-ERC_noise_estimate1})
  gives a sharper estimate and
  we use~(\ref{eq_l1-ERC_noise_estimate1}) for the example from
  digital holography in section \ref{sec_applications}.
\end{remark}

Theorem~\ref{theorem_l1-ERC_noise1} gives a lower bound on the regularization
parameter $\alpha$ to ensure $\supp(u^{\alpha,\varepsilon}) \subset \supp(u\sparse)$.
To guarantee $\supp(u^{\alpha,\varepsilon}) = \supp(u\sparse)$
we need an additional upper bound for $\alpha$.
The following theorem leads to that purpose.

\begin{theorem}[Error estimate]\label{theorem_l1-upperbounds}
  Let the assumptions of theorem~\ref{theorem_l1-ERC_noise1} hold
  and choose $\alpha$ according to~(\ref{eq_l1-parameter_rule1}).
  Then the following error estimate is valid:
  \begin{equation}
    \norm{u\sparse - u^{\alpha,\varepsilon}}_{\ell^\infty}
    \leq (\alpha + \sup\limits_{i\in\Z} |\ip{\eta}{Ke_i}|) \norm{(P_I K^* K P_I)^{-1}}_{\ell^1,\ell^1}.
    \label{eq_l1-upperbound}
  \end{equation}
\end{theorem}

\begin{proof}
  From the assumptions of theorem~\ref{theorem_l1-ERC_noise1}
  we have
  $\supp(u^{\alpha,\varepsilon}) \subset \supp(u\sparse)$.
  From the optimality condition~(\ref{eq_l1-op-cond}) we know that
  for $u^{\alpha,\varepsilon}$
  there is a $w\in\ell^\infty$ with $\norm{w}_{\ell^\infty}\leq1$ such that
  \[
    -K^*(Ku^{\alpha,\varepsilon} - g^\varepsilon) = \alpha \, w.
  \]
  Hence, it holds that
  \begin{align*}
    - P_I K^* K P_I (u^{\alpha,\varepsilon} - u\sparse)
    & = - P_I K^* (K u^{\alpha,\varepsilon} - g)
      = - P_I K^* (K u^{\alpha,\varepsilon} - g^\varepsilon) - P_I K^* \eta \\
    & = \alpha P_I w - P_I K^* \eta.
  \end{align*}
  Since $\norm{w}_{\ell^\infty}\leq 1$,
  with H\"ol\-der's inequality
  we can estimate for all $j\in I$
  \begin{align*}
    |(u^{\alpha,\varepsilon} - u\sparse)_j|
    & = |\ip{u^{\alpha,\varepsilon} - u\sparse}{e_j}|
      = \big|\Ip{(P_I K^* K P_I)^{-1} \big(\alpha P_I w - P_I K^* \eta\big)}{e_j}\big|\\
    & \leq \alpha |\ip{P_I w}{(P_I K^* K P_I)^{-1} e_j}| + |\ip{P_I K^* \eta}{(P_I K^* K P_I)^{-1} e_j}|\\
    & \leq \big(\alpha \norm{w}_{\ell^\infty} + \norm{P_IK^* \eta}_{\ell^\infty} \big)
      \norm{(P_I K^* K P_I)^{-1}}_{\ell^1,\ell^1}\\
    & \leq \big(\alpha + \sup\limits_{i\in\Z} |\ip{\eta}{Ke_i}|\big)
      \norm{(P_I K^* K P_I)^{-1}}_{\ell^1,\ell^1}.
  \end{align*}
  ~\vspace{-6ex}

\end{proof}

\begin{remark}
  \upshape
  Due to the error estimate~(\ref{eq_l1-upperbound}) 
  we achieve a linear convergence rate measured in the $\ell^\infty$ norm.
  In finite dimensions the $\ell^p$ norms are equivalent,
  hence  we also get an estimate for the $\ell^1$ error:
  \[
  \norm{u\sparse - u^{\alpha,\varepsilon}}_{\ell^1} \leq (\alpha +
  \varepsilon\norm{K}) \; |I| \; \norm{(P_I K^* K P_I)^{-1}}_{\ell^1,\ell^1}.
  \]
  Compared to the estimate~\eqref{eq_l1-error_estimate} from
  theorem~\ref{theorem_ell1-convergence-rate}, the quantities $\theta$
  and $\norm{w}$ are not present anymore. The role of $1/c$ is now
  played by $\norm{(P_I K^* K P_I)^{-1}}_{\ell^1,\ell^1}$.  However,
  if upper bounds on $I$ or on its size (together with structural
  information on $K$) is available, our estimate can give a-priori
  checkable error estimates.
\end{remark}

The following theorem gives a 
sufficient condition for the existence of a regularization parameter $\alpha$
which provides exact recovery.
Due to theorem~\ref{theorem_l1-upperbounds}, equation~(\ref{eq_l1-upperbound}), the regularization parameter
should be chosen as small as possible.

\begin{theorem}[Exact recovery condition in the presence of noise]\label{theorem_l1-ERC_noise}
  Let $u\sparse \in \ell^0$ with $\supp (u\sparse) = I$ and 
  $g^\varepsilon = K u\sparse + \eta$ the noisy data with noise-to-signal ratio
  \[
    \nsr :=
    \frac{\sup\limits_{i\in\Z} |\ip{\eta}{Ke_i}|}
         {\min\limits_{i\in I} |u_i\sparse|}.
  \]
  Assume that the operator $K$ is bounded and possesses the FBI property.
  Then the \emph{exact recovery condition in the presence of noise ($\varepsilon$ERC)}
  \begin{equation}\label{eq_l1-ERC_noise2}
    \sup\limits_{i\in I^\complement} \norm{(K P_I)^\dagger K e_i}_{\ell^1}
    < 1 - 2 \nsr \norm{(P_I K^* K P_I)^{-1}}_{\ell^1,\ell^1}
  \end{equation}
  ensures that there is a
  suitable regularization parameter $\alpha$,
  \begin{align}
    \frac{1 + \sup_{i \in I^\complement} \norm{(KP_I)^\dagger Ke_i}_{\ell^1}}
         {1 - \sup_{i\in I^\complement} \norm{(K P_I)^\dagger K e_i}_{\ell^1}}
    &\; \sup\limits_{i\in\Z} |\ip{\eta}{Ke_i}|
    < \alpha \label{eq_l1-parameter_rule2}\\
    \alpha & <
    \frac{\min\limits_{i\in I} |u_i\sparse|}
         {\norm{(P_I K^* K P_I)^{-1}}_{\ell^1,\ell^1}}
    - \sup\limits_{i\in\Z} |\ip{\eta}{Ke_i}|,
    \notag
  \end{align}
  which provides exact recovery of $I$,
  i.e.~the support of the minimizer $u^{\alpha,\varepsilon}$ 
  coincides with $\supp(u\sparse)=I$.
\end{theorem}

\begin{proof}
  The lower bound of the parameter rule~(\ref{eq_l1-parameter_rule2}) ensures that $\supp (u^{\alpha,\varepsilon}) \subset I$.
  With the error estimate~(\ref{eq_l1-upperbound}) we see that
  the upper bound from the parameter rule~(\ref{eq_l1-parameter_rule2})
  guarantees for $j\in I$ that
  \[
    |u_j\sparse| - |u_j^{\alpha,\varepsilon}|
    \leq |u_j\sparse - u_j^{\alpha,\varepsilon}|
    < \min\limits_{i\in I} |u_i\sparse|,
  \]
  hence $|u_j^{\alpha,\varepsilon}|>|u_j\sparse| - \min_{i\in I} |u_i\sparse| \geq 0$ for all $j\in I$.
  The $\varepsilon$ERC~(\ref{eq_l1-ERC_noise2}) ensures that
  the interval of convenient regularization parameters $\alpha$
  resulting from~(\ref{eq_l1-parameter_rule2})
  is not empty.
\end{proof}

\begin{remark}
  \upshape
  Theorem~\ref{theorem_l1-ERC_noise} gives a parameter choice rule that works without
  the quantities $g^\varepsilon_I$ and $u^\varepsilon_I$.
  In theorem~\ref{theorem_tropp},
  the existence of a parameter that ensures exact recovery is guaranteed on the condition
  \begin{equation}
    \sup_{i\in I^\complement} \norm{(K P_I)^\dagger K e_i}_{\ell^1}
    < 1 - 
    \frac{\sup_{i\in I^\complement} |\ip{g^\varepsilon - g^\varepsilon_I}{Ke_i}|}
    {\min_{i\in I} |u^\varepsilon_I (i)|}
    {\norm{(P_I K^* K P_I)^{-1}}_{\ell^1,\ell^1}},
    \label{eq:Tropp_epsERC}
  \end{equation}
  cf. equation~(\ref{eq:para_rule_tropp}).
  The question remains which condition is sharper,~(\ref{eq:Tropp_epsERC})
  or~(\ref{eq_l1-parameter_rule2})?
  The conditions look similar
  with different factors on the right-hand side, however,
  it is not obvious, which condition is shaper:
  Indeed, provided that 
  $\sup_{i\in I^\complement} \norm{(K P_I)^\dagger K e_i}_{\ell^1}
  < 1$ holds, by using equation~(\ref{eq_l1-ERC_noise_estimate1}) 
  we can estimate
  \[
  	\frac{\sup_{i\in I^\complement} |\ip{g^\varepsilon - g^\varepsilon_I}{Ke_i}|}
    {\min_{i\in I} |u^\varepsilon_I (i)|}
	<
    2 \, \frac{\sup_{i\in\Z} |\ip{\eta}{Ke_i}|}
    {\min_{i\in I} |u^\varepsilon_I (i)|}.
  \]
  However, it is not clear which expression, 
  $\min_{i\in I} |u^\varepsilon_I (i)|$ or 
  $\min_{i\in I} |u_i\sparse|$, is smaller:
  Recall,
  $g^\varepsilon_I$ is the orthogonal projection of 
  $g^\varepsilon$ to
  the range of $K P_I$, and $u^\varepsilon_I$ is
  the unique element with $\supp (u^\varepsilon_I) \subset I$
  such that $g^\varepsilon_I = K u^\varepsilon_I$.  
  Now, the noise $\eta$ with $g^\varepsilon = Ku\sparse + \eta$ can
  cause an increase or decrease of the values $u\sparse_j$ and hence,
  the coefficients $u_I^\varepsilon(j)$, $j\in I$ can be larger or
  smaller.
\end{remark}  

The results of theorem~\ref{theorem_l1-upperbounds}
and~\ref{theorem_l1-ERC_noise} can be rephrased as follows: If the
regularization parameter $\alpha(\varepsilon)$ is chosen according to
(\ref{eq_l1-parameter_rule2}) and fulfills $\alpha\asymp\varepsilon$,
then $\norm{u^{\alpha,\varepsilon}-u\sparse}_{\ell^1}=\bigO(\varepsilon)$ and
the support of $u^{\alpha,\varepsilon}$ coincides with that of
$u\sparse$. Indeed, this can be interpreted as convergence in the space
$\ell^0$ with respect to the following topology.
\begin{definition}
  We equip the spaces $\R^n$ with the Euclidean topology and consider
  them ordered by inclusion $\R^n\subset\R^{n+1}$ in the natural way.
  Then an absolutely convex and absorbent subset $U$ of $\ell^0$ is
  called a neighborhood of $0$ if set $U\cap\R^n$ is open in
  $\R^n$ for any $n$. The topology $\tau$ on $\ell^0$ which is
  generated by the local base of these neighborhoods is called the
  \emph{topology of sparse convergence}.
\end{definition}

The definition above says that the topology of sparse convergence is
generated as strict inductive limit of the spaces $\R^n$. The space
$\ell^0$ is turned into a complete locally convex vector space with
this topology. A sequence $(u^n)$ in $\ell^0$ converges to $u$ in this
topology if there is a finite set $I\subset\N$ such that $\supp
u^n\subset I$, for all $n\in\N$, and the sequence $u^n$ converges componentwise. As a
strict inductive limit of Fr\'echet spaces, $(\ell^0,\tau)$ is also
called an LF-space and is known to be not normable,
see~\cite{Dieudonne1949,Narici1985}. The topology of sparse
convergence resembles the topology on the space of test functions
$\mathcal{D}(\Omega)$ in distribution theory. (This correspondence can
be pushed a little bit further by observing that the dual space of
$(\ell^0,\tau)$ is the space of all real valued sequences and plays
the rule of the space of distributions. We will not pursue this
similarity further here.)
\begin{corollary}[Convergence in $\ell^0$]
  In the situation of theorem~\ref{theorem_l1-ERC_noise} assume that
  $\alpha$ is chosen to fulfill the parameter choice
  rule~\eqref{eq_l1-parameter_rule2}. Then $u^{\alpha,\varepsilon}\to
  u\sparse$ in $\ell^0$ in the topology of sparse convergence.
\end{corollary}

In fact, the parameter choice rule~(\ref{eq_l1-parameter_rule2}) is
not an a priori parameter rule $\alpha=\alpha(\varepsilon)$, since it
depends on the noise $\eta$ and on unknown quantities such as $I$ and
$\min_i |u\sparse_i|$.  However, the term $\sup_{i\in\Z}
|\ip{\eta}{Ke_i}|$ is related to the noise level and it can be
estimated by $\varepsilon\norm{K}$,
cf.~remark~\ref{remark_l1-noise-level}. The term $\min_i
|u\sparse_i|$, i.e.~the smallest non-zero entry in the unknown
solution, may be estimated from below in several applications.  Due to
the expressions $\norm{(P_I K^* K P_I)^{-1}}_{\ell^1,\ell^1}$ and
$\sup_{i\in I^\complement} \norm{(K P_I)^\dagger K e_i}_{\ell^1}$, the
$\varepsilon$ERC~(\ref{eq_l1-ERC_noise2}) is hard to evaluate,
especially since the support $I$ is unknown.  Therefore, we
follow\cite{Dossal2005} and give another sufficient recovery condition
which is on the one hand weaker in the sense that it is easier to
satisfy (and implies the $\varepsilon$ERC and hence, is less powerful
than the ERC) and on the other hand is easier to evaluate in practice
since it only depends on inner products of images of $K$ restricted to
$I$ and $I^\complement$.  For the sake of an easier presentation we
define according to~\cite{Gribonval2008,Dossal2005}
\[
  \CI  := \sup\limits_{i\in I} \sum\limits_{\sumstack{j\in I}{j\neq i}} |\ip{K e_i}{K e_j}|
  \qquad \mbox{and} \qquad
  \CIC := \sup\limits_{i\in I^\complement} \sum\limits_{j\in I} |\ip{K e_i}{K e_j}|.
\]

\begin{theorem}[Neumann exact recovery condition in the presence of noise]\label{theorem_l1-NeumannERC_noise}
  Let $u\sparse \in \ell^0$ with $\supp (u\sparse) = I$ and 
  $g^\varepsilon = K u\sparse + \eta$ the noisy data 
  with noise-to-signal ratio $\nsr$.
  Assume that the operator norm of $K$ is bounded by 1 and that $K$ possesses the FBI property.
  Then the \emph{Neumann exact recovery condition in the presence of noise (Neumann $\varepsilon$ERC)}
  \begin{equation}
    \CI + \CIC
    < \min\limits_{i\in I} \norm{K e_i}_{\hilbert_2}^2 \, - \, 2 \nsr
    \label{eq_l1-00-NeumannERC_noise2}
  \end{equation}
  ensures that there is a
  suitable regularization parameter $\alpha$,
  \begin{align}
    & \frac{\min\limits_{i\in I} \norm{K e_i}_{\hilbert_2}^2  - \CI + \CIC}
         {\min\limits_{i\in I} \norm{K e_i}_{\hilbert_2}^2  - \CI - \CIC}
    \; \sup\limits_{i\in\Z} |\ip{\eta}{Ke_i}|
    < \alpha \label{eq_l1-00-Neumannparameter_rule2}\\
    & \qquad\qquad\qquad \alpha <
    \big(\min\limits_{i\in I} \norm{K e_i}_{\hilbert_2}^2 - \CI\big)
    \min\limits_{i\in I} |u_i\sparse| - \sup\limits_{i\in\Z} |\ip{\eta}{Ke_i}|,
    \notag
  \end{align}
  which provides exact recovery of $I$,
  i.e.~the support of $u^{\alpha,\varepsilon}$ 
  coincides with $\supp(u\sparse)=I$.
\end{theorem}

\begin{proof}
  For the deduction of conditions~(\ref{eq_l1-00-NeumannERC_noise2}) and~(\ref{eq_l1-00-Neumannparameter_rule2}) 
  from conditions~(\ref{eq_l1-ERC_noise2}) and~(\ref{eq_l1-parameter_rule2}),
  respectively, 
  one splits the operator $P_I K^* K P_I$ into diagonal and off-diagonal
  and uses $\norm{K}\le 1$ and a Neumann series expansion for its inverse,
  following the techniques from~\cite{Gribonval2008,Dossal2005}.
\end{proof}

\begin{remark}
  \upshape
  By the assumptions of theorem~\ref{theorem_l1-NeumannERC_noise}, the operator norm of $K$ is bounded by 1,
  i.e. $\norm{Ke_i}_{\hilbert_2}\leq 1$ for all $i\in\Z$.
  Hence, to ensure the Neumann $\varepsilon$ERC~(\ref{eq_l1-00-NeumannERC_noise2}),
  one has necessarily for the noise-to-signal ratio $\nsr<1/2$.
  For a lot of examples one can normalize $K$, so that $\norm{K e_i}_{\hilbert_2} = 1$ holds for all $i\in\Z$.
  We do this for the example from digital holography in
  section \ref{sec_applications}.
  In this case the 
  Neumann $\varepsilon$ERC~(\ref{eq_l1-00-NeumannERC_noise2}) reads as
  \[
    \CI + \CIC
    < 1 \, - \, 2 \nsr.
  \]
  This condition coincides with 
  the result presented in~\cite{Denis2009c} for the orthogonal matching pursuit.
\end{remark}

\begin{remark}\label{remark_fuchs}
  \upshape
  We remark that the correlations $\CI$ and $\CIC$ can be estimated from
  above, with $N:= \norm{u\sparse}_{\ell^0}$, by
  \[
    \CI  \leq (N-1) \, \mu
    \qquad \mbox{and} \qquad
    \CIC \leq N \, \mu.
  \]
  Consequently, the exact recovery condition in terms of the coherence parameter $\mu$
  can easily be deduced from conditions~(\ref{eq_l1-00-NeumannERC_noise2}) and~(\ref{eq_l1-00-Neumannparameter_rule2}).
  This will result in Fuchs' exact recovery condition from~\cite{fuchs2005exactsparse}.
\end{remark}

\section{Relations between recovery conditions and the source
  condition}
\label{sec:relat-betw-recov}

In this section we compare the different conditions which have been used.
As we have seen in theorems~\ref{theorem_ell1-convergence-rate}
and~\ref{theorem_l1-upperbounds} both the
SC~\eqref{eq_l1-soucecond-convergence rate} and the ERC~\eqref{eq_ERC}
lead to a linear convergence rate under an appropriate parameter
choice rule. However, the latter also leads to exact recovery.  One
should note that the ERC and the SC are crucially different is some sense:
The ERC is a \emph{uniform} condition in the sense that it uses a
given support $I$ and hence, also leads to a result which holds for
all vectors with that support. The SC on the other hand depends on a
particular sign pattern $\Sign(u\sparse)$ and hence, leads to a result
which holds for all vectors with that sign pattern. We may hence
strengthen the SC to a ``uniform source condition'' (uniform SC)  as follows:
\begin{equation*}
  \label{eq_uniform_SC}
  \text{for all }\ u\sparse \text{ with } \supp u\sparse\subset I\ \text{ there exists }\ w\in\hilbert_2 \text{ such that }\  K^*w \in \Sign(u\sparse).
\end{equation*}
As it turns out, the ERC does not only imply the uniform SC but even a
``uniform strict source condition'':
\begin{proposition}[ERC $\Rightarrow$ uniform strict SC]
  Let $I$ be finite and let $KP_I$ be injective. Then the
  ERC~\eqref{eq_ERC} implies the following \emph{uniform strict SC}:
  \begin{equation}
  \label{eq_uniform_strict_SC}
  \text{for all }\ u\sparse \text{ with } \supp(u\sparse)\subset I\ \text{ there exists }\ w\in\hilbert_2:
  \begin{cases}
    P_IK^*w = P_I\sign(u\sparse)\\
    \norm{P_{I^\complement}K^*w}_{\ell^\infty} < 1.
  \end{cases}
\end{equation}
\end{proposition}
\begin{proof}
  Let $u\sparse$ be such that $\supp(u\sparse)\subset I$. Since $K
  P_I$ is injective, the operator $P_IK^*:\hilbert_2 \to \ell^2(I)$ is
  surjective and hence, the equation
  \[
  P_IK^* w = P_I\sign(u\sparse)
  \]
  has a solution which can be expressed as
  \[
  w = (P_IK^*)^\dagger P_I K^* w = (P_IK^*)^\dagger P_I\sign(u\sparse).
  \]
  Now it remains to check, that for $j\notin I$ it holds that
  $|\ip{K^*w}{e_j}|< 1$: With the H\"older inequality
  and the ERC it follows that
  \begin{align*}
    |\ip{K^*w}{e_j}| & = |\ip{K^*(P_IK^*)^\dagger P_I\sign(u\sparse)}{e_j}|\\
    &= |\ip{P_I\sign(u\sparse)}{(KP_I)^\dagger K e_j}|\\
    &\leq \norm{P_I\sign(u\sparse)}_{\ell^\infty}\norm{(KP_I)^\dagger K
      e_j}_{\ell^1}<1.
  \end{align*}
  Finally, $|\ip{K^*w}{e_j}|\to 0$ for $j\to\infty$ and hence
  $\norm{P_{I^\complement}K^*w}_{\ell^\infty} < 1$.
\end{proof}
It should be noted that a similar result appears in~\cite[Theorem
4.7]{Grasmair_suf_nec_ell1}. There it is shown that a linear
convergence rate for the minimizers $u^{\alpha,\varepsilon}$ already
implies that the strict SC holds and hence, by
theorem~\ref{theorem_l1-upperbounds} ERC implies strict SC.

% \begin{example}[strict SC $\nRightarrow$ ERC]
%   Consider
%   \[
%   K =
%   \begin{bmatrix}
%     0.5 & 0.5 & 0\\
%     0   & 0.5 & 0.5
%   \end{bmatrix}
%   \text{ and }
%   u\sparse =
%   \begin{bmatrix}
%     0\\ 1 \\ 1
%   \end{bmatrix}
%   \]
%   Then the strict SC is fulfilled since for
%   \[
%   w =
%   \begin{bmatrix}
%     0\\ 2
%   \end{bmatrix}
%   \]
%   it holds that $ K^*w = u\sparse$.  However, ERC is not
%   fulfilled: We have
%   \[
%   (K P_I)^\dagger =
%   \begin{bmatrix}
%     0 & 0\\
%     2 & 0\\
%     -2 & 2
%   \end{bmatrix}
%   \]
%   and hence
%   \[
%   \norm{(K P_I)^\dagger Ke_1}_1 = 2.
%   \]  
% \end{example}

Another important condition in the context of sparse recovery is the
so called \emph{null space property} (NSP).  An operator $K:\ell^2\to
\hilbert_2$ is said to have the NSP for the set
$I\subset \N$, if for any $u\in \kr K$, $u\neq 0$ it holds that
\begin{equation}
  \label{eq:NSP}
  \norm{P_Iu}_{\ell^1} < \norm{P_{I^\complement}u}_{\ell^1}.
\end{equation}
The importance of the NSP comes from the following theorem on the
performance of $\ell^1$-minimization:
\begin{theorem}[{\cite[Thm.~2, Thm.~3]{gribonval2007sparsenessmeasure}}]
  \label{thm:NSP}
  Any vector $u\sparse$ with $\supp u\sparse \subset I$ is the unique
  solution of
  \[
  \min_u \norm{u}_{\ell^1}\ \text{ s.t. }\ Ku = Ku\sparse
  \]
  if and only if $K$ fulfills the NSP for the set $I$.
\end{theorem}

However, the NSP is implied by the uniform strict SC:
\begin{proposition}
  \label{prop:usSC-implies-NSP}
  The uniform strict SC~\eqref{eq_uniform_strict_SC} implies the NSP
  \eqref{eq:NSP}.
\end{proposition}
\begin{proof}
  For any $u\in \kr K$ and any $v\in\hilbert_2$ it holds that
  \[
  0 = \ip{Ku}{v} = \ip{u}{K^* v}.
  \]
  Now we define $u\sparse$ by
  \[
  i\in I:\ \sign(u\sparse_i) = -\sign(u_i),\qquad i\notin I:
  u\sparse_i =0.
  \]
  Due to~\eqref{eq_uniform_strict_SC} we can find $w$ such that $K^*w
  \in\Sign(u\sparse)$ and moreover
  $\norm{P_{I^\complement}w}_{\ell^\infty}<1$. Using this $w$ instead of $v$,
  we get from the definition of $u\sparse$ and the H\"older inequality
  \[
  0 = \ip{u}{K^*w} = \sum_{i\in I}u_i\sign(u\sparse_i) + \sum_{i\notin
    I} u_i (K^*w)_i< -\norm{P_Iu}_{\ell^1} + \norm{P_{I^\complement}u}_{\ell^1}
  \]
  which shows the assertion.
\end{proof}
The fact that the strict SC is an important condition in this context
was already observed in~\cite[Section
II]{candes2005decoding}. Combining their argumentation there with
theorem~\ref{thm:NSP} one obtains another proof of
proposition~\ref{prop:usSC-implies-NSP}.

Since there are plenty of conditions which are related to the
performance of $\ell^1$-minimization, we end this section with an
illustration of the implications between different conditions in the
context of this paper. First, the obvious implication between the
different ``ERCs'':
% \[
% \begin{array}{ccccccc}
%   & & & & \varepsilon\text{ERC}& & \\
%   & & & & \Downarrow & & \\
%   \text{Neumann $\varepsilon$ERC} & \Rightarrow & \text{Neumann ERC} & \Rightarrow & \text{ERC} & \Rightarrow & \text{strict SC}\\
%   & & & & & & \Downarrow\\
%   & & & & \text{uniform SC} & \Rightarrow & \text{SC}\\
%   & & & & \Downarrow &  & \\
%   & & & & \text{NSP} &  & \\
% \end{array}
% \]
\[
\begin{array}{ccccccccc}
   & & \varepsilon\text{ERC} & & \\
   & \Nearrow & & \Searrow & \\
  \text{Neumann $\varepsilon$ERC} & & & & \text{ERC} \\
   & \Searrow & & \Nearrow & \\
   & & \text{Neumann ERC} & & 
\end{array}
\]
And then the relation of ERC, SC and NSP:
\[
\begin{array}{ccccc}
    & & & &\text{SC}\\
    & & & \Nearrow& \\
    \text{ERC + $KP_I$ injective}& \Rightarrow & \text{uniform strict SC} & &\\
   & & & \Searrow &\\
    & & & &\text{NSP}\\
%   & &  & & & & & \Nearrow&\\
%   & &  & & \text{NSP}& \Leftarrow & \text{uniform SC}& & 
\end{array}
\]
\begin{remark}
  \label{rem_ERC_NSP}
  The implication ``ERC + $KP_I$ injective $\Rightarrow$ NSP'' has
  already been observed in~\cite[remark
  4]{gribonval2007sparsenessmeasure}. Moreover,~\cite[example
  1]{gribonval2007sparsenessmeasure} shows that the converse
  implication does not hold. We postpone further investigation of
  converse implications to future work.
\end{remark}

\section{Application of exact recovery conditions to digital holography}
\label{sec_applications}

To apply the Neumann $\varepsilon$ERC~(\ref{eq_l1-00-NeumannERC_noise2}),
one has to know the support $I$.
In this case, there would be no need to apply complex reconstruction methods. One may just
solve the restricted least squares problem.
For deconvolution problems, however, with a certain prior knowledge,
it is possible to evaluate
the Neumann $\varepsilon$ERC~(\ref{eq_l1-00-NeumannERC_noise2})
a priori,
especially when the support $I$ is not known exactly.

In the following we use the Neumann $\varepsilon$ERC~(\ref{eq_l1-00-NeumannERC_noise2}) exemplarily
for an inverse convolution problem
as it is used in digital holography of particles~\cite{soulez2007holography,Denis2009d}.
The presentation relies on~\cite{Denis2009c} and we reproduce it here for the sake of completeness
in a compact style.
In digital holography, the hologram corresponds to the diffraction patterns of the illuminated particles.
The hologram is recorded digitally on a charge-coupled device (CCD), from 
the diffraction patterns the size and the distribution of particles are reconstructed.

We consider the case of spherical particles,
which is of significant interest in
applications such as fluid mechanics.
We model the particles $j\in\{1,\ldots,N\}$
as opaque disks $B_r(\cdot-x_j,\cdot-y_j,\cdot-z_j)$
with center $(x_j,y_j,z_j)\in\R^3$ and radius $r$.
Hence the source $f\sparse$ is given as a sum of characteristic functions
\[
  \textstyle
  f\sparse = \sum\limits_{j=1}^N u_j\sparse \, \chi_{B_r}(\cdot-x_j,\cdot-y_j,\cdot-z_j)
    =: \sum\limits_{j=1}^N u_j\sparse \, \chi_j.
\]
The real values $u_j\sparse$ are amplitude factors of the diffraction pattern
that in practice depend on experimental parameters.

The forward operator $K:\ell^2 \to L^2(\R^2)$, which maps the coefficients $u_j$ to
the corresponding digital hologram, 
is well modeled by a bidimensional convolution $\convbi$ with respect to $(x,y)$.
In the following $\iota$ represents the imaginary unit.
Let $h_{z_j}$ constitute the Fresnel function defined by
\[
  h_{z_j}(x,y) = \frac{1}{\iota \lambda z_j} \exp \Big(\iota \frac{\pi}{\lambda z_j} \norm{R}^2 \Big),
  \qquad 
  \mbox{with }
  R:=(x,y).
\]
With that, the hologram of a particle at position $(x_j,y_j,z_j)$
and hence the corresponding
operator response $K e_j$ has the following form~\cite{soulez2007holography}:
\begin{equation}\label{eq_atom_fresnel_1}
  \big(K e_j\big) (x,y)
  := \frac{2}{\pi r^2}
  \chi_{B_r}(x-x_j,y-y_j) \;\convbi\;  \text{Re}\big(h_{z_j} (x-x_j,y-y_j) \big).
\end{equation}
The factor $2/ (\pi r^2)$ assures $K e_j$ to be unit-normed, cf.~\cite{Denis2009c}.

%In practice, the function $\chi_{B_r}$ has a bounded and small support with respect to $\sqrt{\lambda z}$,
%where $\lambda$ is the wavelength of the laser.
%Then the unit-normed forward operator $K:\ell^2 \to L^2(\R^2)$, which maps the coefficients $u_j$ to
%the corresponding digital hologram, 
%is well modeled by
%\begin{equation}\label{eq_atom_fresnel_2}
%  \big(K e_j\big) (x,y)
%  = \Big(\frac{2}{\pi}\Big)^{\frac{1}{2}} \frac{1}{\norm{R_j}}
%    J_1\Big(\frac{2\pi r}{\lambda z_j}\norm{R_j}\Big)
%    \sin \Big(\frac{\pi}{\lambda z_j} \norm{R_j}^2\Big),
%\end{equation}
%where $R_j:=(x-x_j,y-y_j)$ and $J_1$ denotes the first kind Bessel function of order one, see~\cite{Denis2009c}.

The first step to evaluate the
Neumann $\varepsilon$ERC~(\ref{eq_l1-00-NeumannERC_noise2})
is to calculate the correlation 
$|\ip{K e_i}{K e_j}|$ with distance $\varrho_{j,i} := (x_j-x_i,y_j-y_i)$.
In the following we assume that all particles are located in a plane parallel to the detector,
i.e.~$z:=z_i$ is constant for all $i$.
In~\cite{Denis2009c} it has been shown that the correlation in digital holography
can be estimated by the following majorizing function $\majo:\R^+ \to \R$, 
with known constants $b_L \approx 0.6748$, $c_L \approx 0.7857$ and 
$C(d)$ denoting the area of the intersection of two circles with radius $r$ and distance $d$
%$C(\rho) := (\chi_{B_r} * \chi_{B_r}) (\rho)$,
\begin{align}
  |\ip{K e_i&}{K e_j}| (\varrho_{j,i})
  \leq
    \majo(\norm{\varrho_{j,i}})\notag\\
  & :=
    \frac{C(\norm{\varrho_{j,i}})}{\pi r^2}
    + \frac{1}{4}
    \min \Big\{b_L^2,
    c_L^2 \Big(\frac{\lambda z}{2\pi r}\Big)^{\frac{2}{3}} \norm{\varrho_{j,i}}^{-\frac{2}{3}} \Big\}
    \min \Big\{1,
    \frac{2\lambda z}{\pi} \norm{\varrho_{j,i}}^{-2} \Big\},
  \label{eq_majo}
\end{align}
which is monotonically decreasing in $\norm{\varrho_{j,i}}$, cf.~\cite{Denis2009c}.

With the estimate~(\ref{eq_majo}), we come to a resolution bound for droplets jet reconstruction,
as e.g.~used in~\cite{soulez2007holography}.
Here monodisperse droplets (i.e.~they have the same size, shape and mass)
were generated and emitted on a strait line parallel to the detector plane.
This configuration eases the computation of the $\varepsilon$Neumann ERC.
We define that the particles are located
at some grid points
\[
  \Delta\Z:=\set{i\in\Z}{i/\Delta \in\Z},
\]
where the parameter $\Delta$ describes the grid refinement.
Assume that the particles have the minimal distance
\[
  \rho := \min\limits_{i,j\in\supp (u\sparse)} \norm{\varrho_{j,i}}\in\Delta\N,
\]
then the sums of correlations $\CI$ and $\CIC$ 
can be estimated from above. 
W.l.o.g. we fix one particle at the origin and estimate with the worst case that the 
other particles appear at a distance of $j\rho$ to the origin, with $-\lfloor N/2 \rfloor \leq j \leq \lfloor N/2 \rfloor$.
Then, 
for $\rho>\Delta$ we get
\begin{align*}
  \CI =
  \sup\limits_{i\in I} \sum\limits_{\sumstack{j\in I}{j \neq i}} & |\ip{K e_i}{K e_j}|
    \leq 2 \; \sum\limits_{j=1}^{\lfloor N/2 \rfloor} \majo(j\rho),\\
  \CIC =
  \sup\limits_{i\in I^\complement} \sum\limits_{j\in I} & |\ip{K e_i}{K e_j}|
    \leq \sup\limits_{\sumstack{i\in\Delta\Z}{\Delta\leq i \leq \rho-\Delta}} \;
    \sum\limits_{j=-\lfloor N/2 \rfloor}^{\lfloor N/2 \rfloor}
    \majo(|j \rho-i|).
\end{align*}

Consequently, we can formulate an estimate for the Neumann $\varepsilon$ERC~(\ref{eq_l1-00-NeumannERC_noise2}).

\begin{proposition}[Neumann $\varepsilon$ERC for Fresnel-convolved characteristic functions]\label{theorem_fresnel}
  An estimate from above for the Neumann $\varepsilon$ERC~(\ref{eq_l1-00-NeumannERC_noise2})
  for characteristic functions convolved with the real part of the Fresnel kernel
  is for $\rho>\Delta$
  \begin{equation}
    2 \sum\limits_{j=1}^{\lfloor N/2 \rfloor} \majo(j\rho)
    + \sup\limits_{1\leq i < \tfrac{\rho}{\Delta}} \;
      \sum\limits_{j=-\lfloor N/2 \rfloor}^{\lfloor N/2 \rfloor}
      \majo(|j \rho-i\Delta|)
    < 1 - 2\nsr.
    \label{eq_theorem_fresnel}
  \end{equation}
  This means, that there is a regularization parameter $\alpha$ which allows exact recovery 
  of the support
  with the $\ell^1$-penalized Tikhonov regularization,
  if the above condition is fulfilled.
\end{proposition}

\begin{remark}
  \upshape
  In comparison,
  the exact recovery condition in terms of the coherence parameter 
  $\mu =\sup_{i\neq j}|\ip{K e_i}{K e_j}|(\varrho_{j,i}) \leq \sup_{i\neq j}\majo(\norm{\varrho_{j,i}}) = \majo (\Delta)$
  according to Fuchs' exact recovery condition from~\cite{fuchs2005exactsparse}
  appears as
  \[
    (2N -1) \; \majo(\Delta) < 1 - 2\nsr,
  \]
  see remark~\ref{remark_fuchs}.
  This condition 
  is significantly worse than~(\ref{eq_theorem_fresnel}), since asymptotically it holds that
  $\majo(\norm{\varrho_{j,i}}) \sim \norm{\varrho_{j,i}}^{-\frac{8}{3}}$.
\end{remark}

Condition~(\ref{eq_theorem_fresnel}) of proposition~\ref{theorem_fresnel}
seems not to be easy to handle due to the
upper bound $\majo$ from~(\ref{eq_majo}).
However, in practice all parameters are known, and one can compute a bound
via approaching from large $\rho$.
As soon as the sum is smaller than $1-2\nsr$, it is guaranteed that the $\ell^1$-penalized
Tikhonov regularization can recover exactly.

We apply the Neumann $\varepsilon$ERC~(\ref{eq_theorem_fresnel}) to simulated data of
droplets jets. For the simulation we use a red laser of wavelength $\lambda=0.6328$\textmu m and
a distance of $z=$ 200mm from the camera.
The particles have a diameter of $100$\textmu m and for the corresponding
grid we choose a refinement of $25$\textmu m. Those parameters correspond to that of the 
experimental setup used in~\cite{soulez2007holography,soulez2007holography2}.

After applying the digital holography model,
we add Gaussian noise of different noise levels and in each case of zero mean.
For the coefficients $u\sparse_i$, we choose a setting which implies $u\sparse_i \approx 10$ for all $i\in I$.
Figure~\ref{bild_holo1} and~\ref{bild_holo2} show simulated holograms with
different distances $\rho$ and different noise-to-signal ratios $\nsr$.
For all noisy examples in the right columns of  figure~\ref{bild_holo1} and figure~\ref{bild_holo2} it manually was possible to find a regularization parameter $\alpha$ so that all the particles were recovered exactly.
For minimization of the Tikhonov functional we used the iterated soft-thresholding
algorithm~\cite{daubechies2003iteratethresh}.

However, only for the image in figure~\ref{bild_holo1} ($\rho\approx721$\textmu m)
condition~(\ref{eq_theorem_fresnel}) of proposition~\ref{theorem_fresnel} holds,
hence the existence of a suitable regularization parameter was guaranteed.
For the examples in figure~\ref{bild_holo2} the existence of a regularization parameter which ensures exact recovery cannot be shown, i.e.~condition~(\ref{eq_theorem_fresnel}) of proposition~\ref{theorem_fresnel} is not valid.
In the image on top of figure~\ref{bild_holo2},
the particles have
a too small distance to each other ($\rho\approx360$\textmu m), and
even for the noiseless case condition~(\ref{eq_theorem_fresnel}) is not fulfilled.
The image at the bottom of figure~\ref{bild_holo2} ($\rho\approx721$\textmu m) was manipulated with unrealistically huge noise,
so that condition~(\ref{eq_theorem_fresnel}) is violated, too.

\begin{figure}[ht]
  \begin{tabular}{ccc}
    \includegraphics[width=0.35\linewidth]{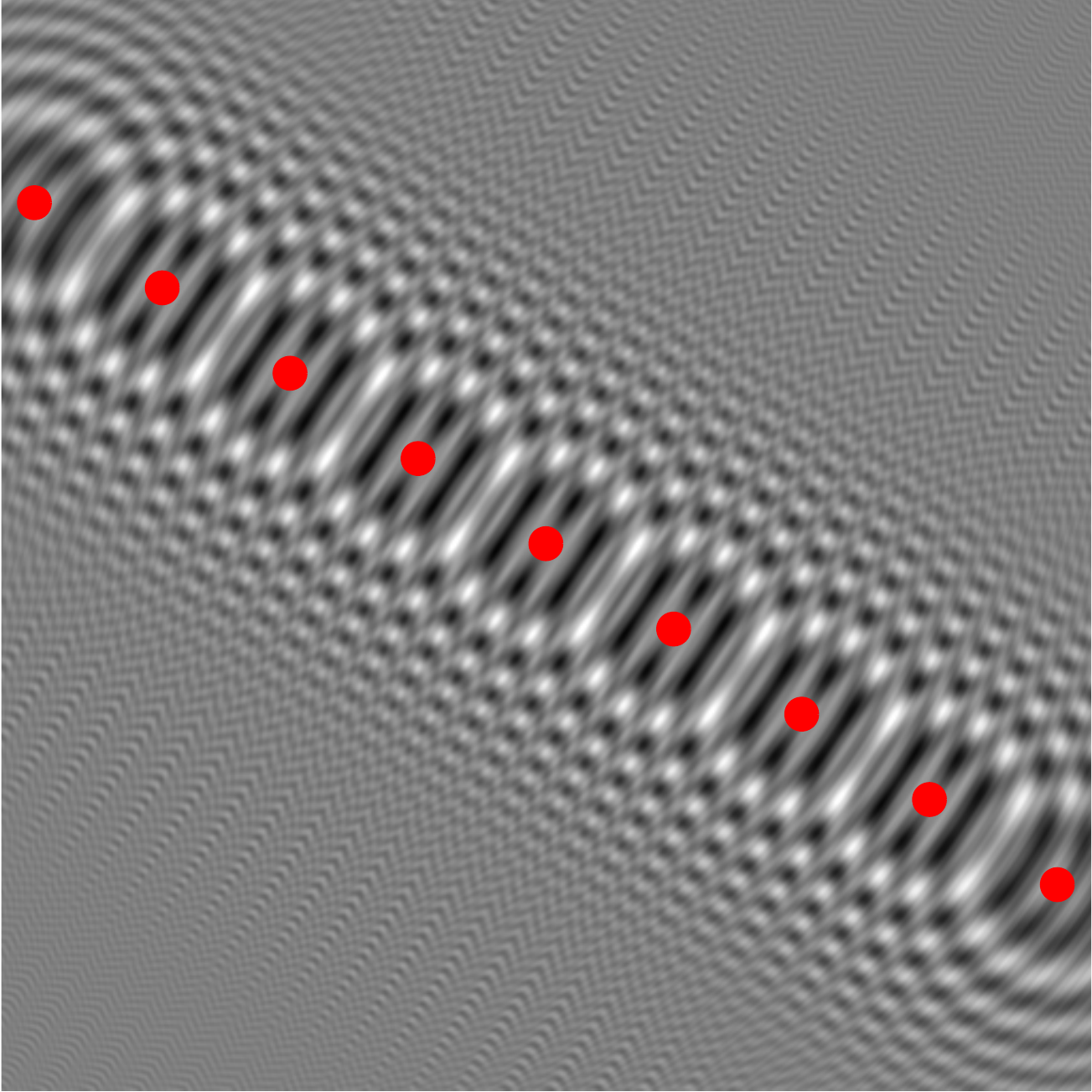}
    & \includegraphics[width=0.35\linewidth]{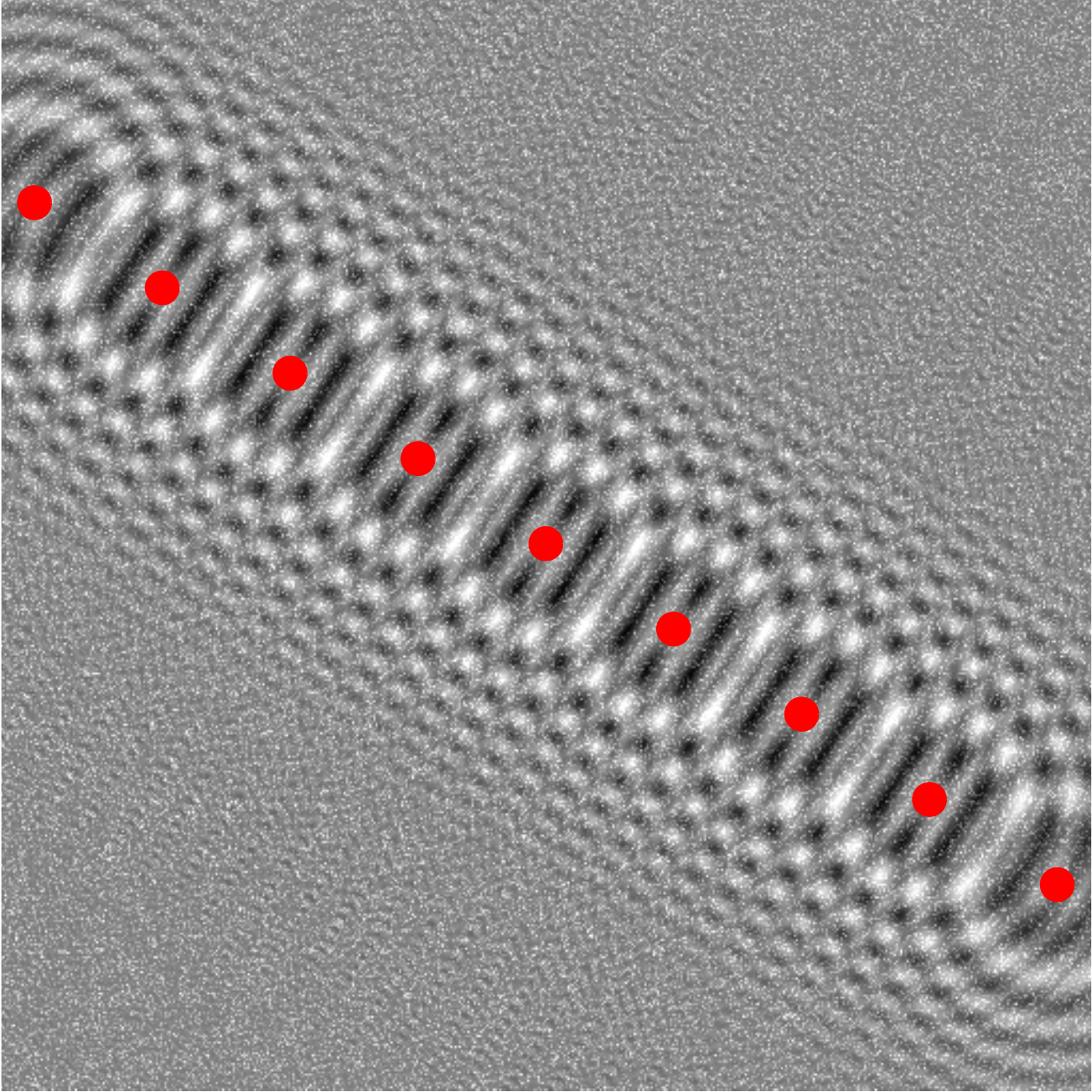}
    &
    \begin{tikzpicture}[xscale=1,yscale=0.75]
      \draw[top color=white,bottom color=black,use as bounding box] (0,0) rectangle (0.5,6.1);
          \foreach \x in {0.0,1.525,3.05,4.575,6.1} %{1.3,1.15,1.0,0.85,0.7}
      \draw (0,\x) -- (0.7,\x) node[right] {\footnotesize {\FPeval\result{round((6/61*{\x} + 0.7):2)}$\FPprint\result$}};
    \end{tikzpicture}\\
    \vspace{-1.2ex}\\
  \end{tabular}
    \caption{Simulated holograms of spherical particles.
           On the left-hand side the noiseless signal is displayed.
           For reconstruction, the noisy signal
           on the right-hand side is used.
           The dots correspond to the true location of the particles.
           The existence of a suitable regularization parameter is
           guaranteed by condition~(\ref{eq_theorem_fresnel}) 
           of proposition~\ref{theorem_fresnel} and
           hence the $\ell^1$-penalized Tikhonov regularization
           recovered all particles exactly.}
  \label{bild_holo1}
\end{figure}
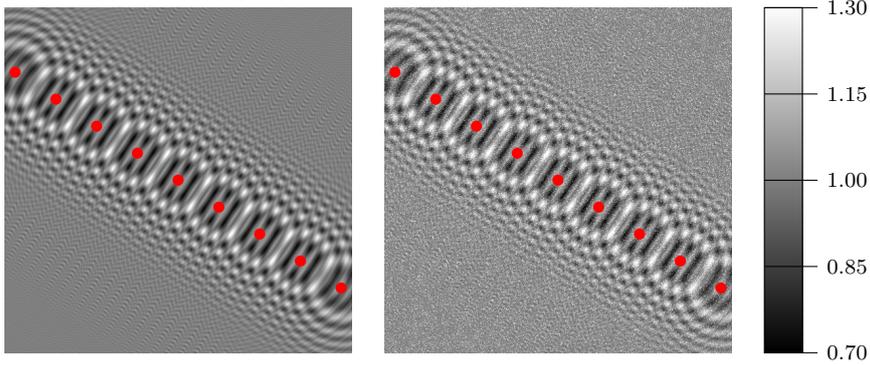

\begin{figure}[ht]
  \begin{tabular}{ccc}
    \includegraphics[width=0.35\linewidth]{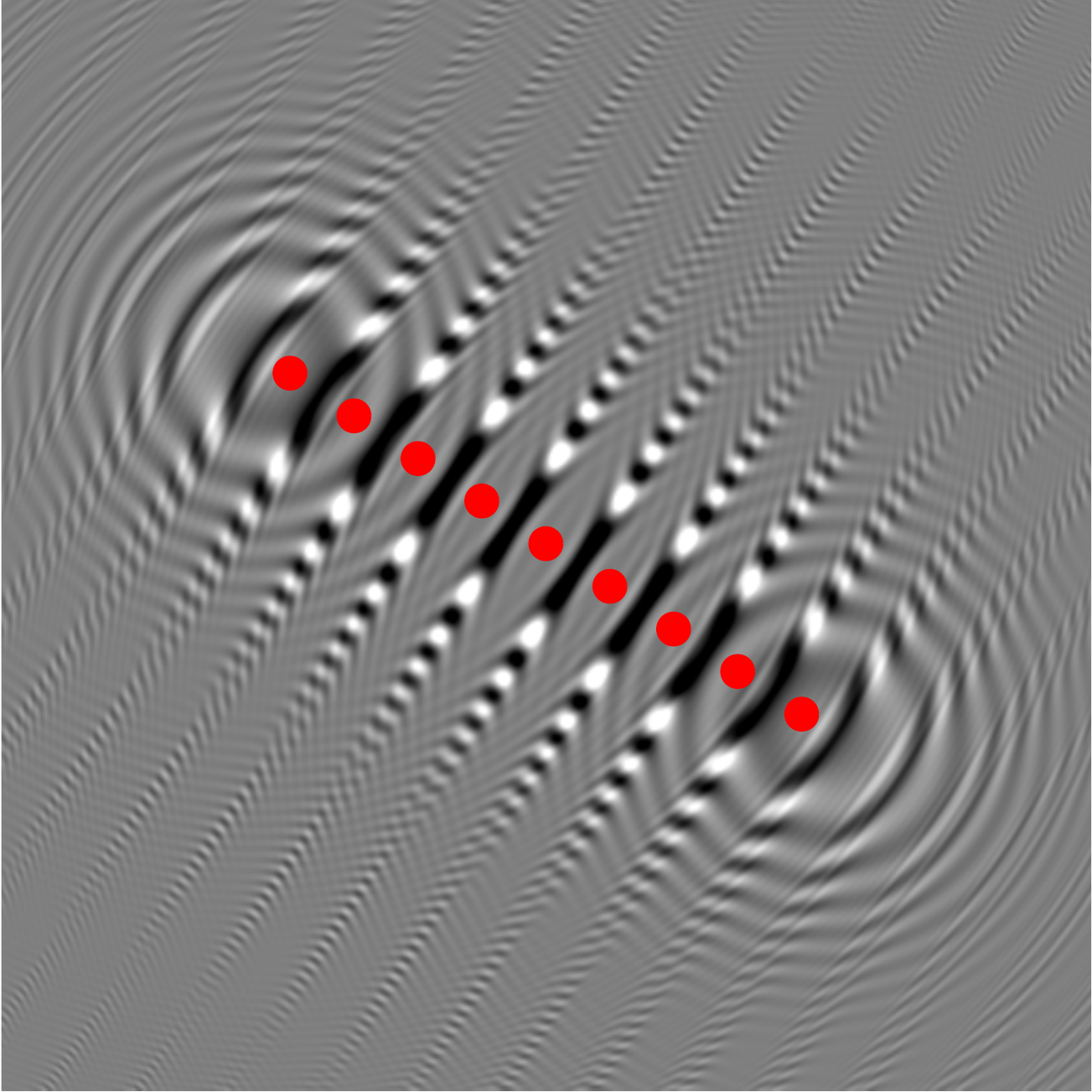}
    & \includegraphics[width=0.35\linewidth]{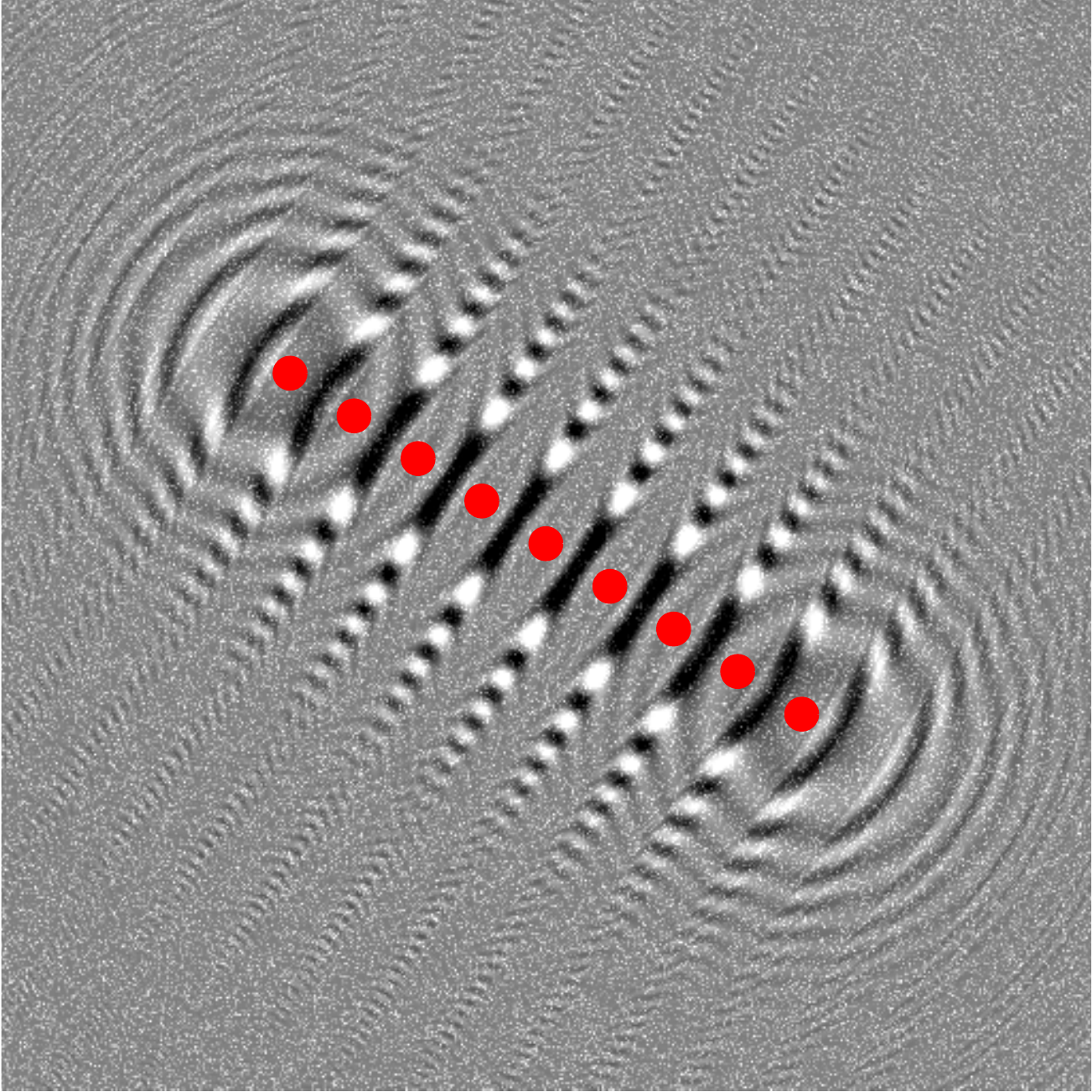}
    &
    \begin{tikzpicture}[xscale=1,yscale=0.75]
      \draw[top color=white,bottom color=black,use as bounding box] (0,0) rectangle (0.5,6.1);
          \foreach \x in {0.0,1.525,3.05,4.575,6.1} %{1.3,1.15,1.0,0.85,0.7}
      \draw (0,\x) -- (0.7,\x) node[right] {\footnotesize {\FPeval\result{round((6/61*{\x} + 0.7):2)}$\FPprint\result$}};
    \end{tikzpicture}\\
    \vspace{-1.2ex}\\
    \includegraphics[width=0.35\linewidth]{./bilder/DH_low-res_no-noise.pdf}
    & \includegraphics[width=0.35\linewidth]{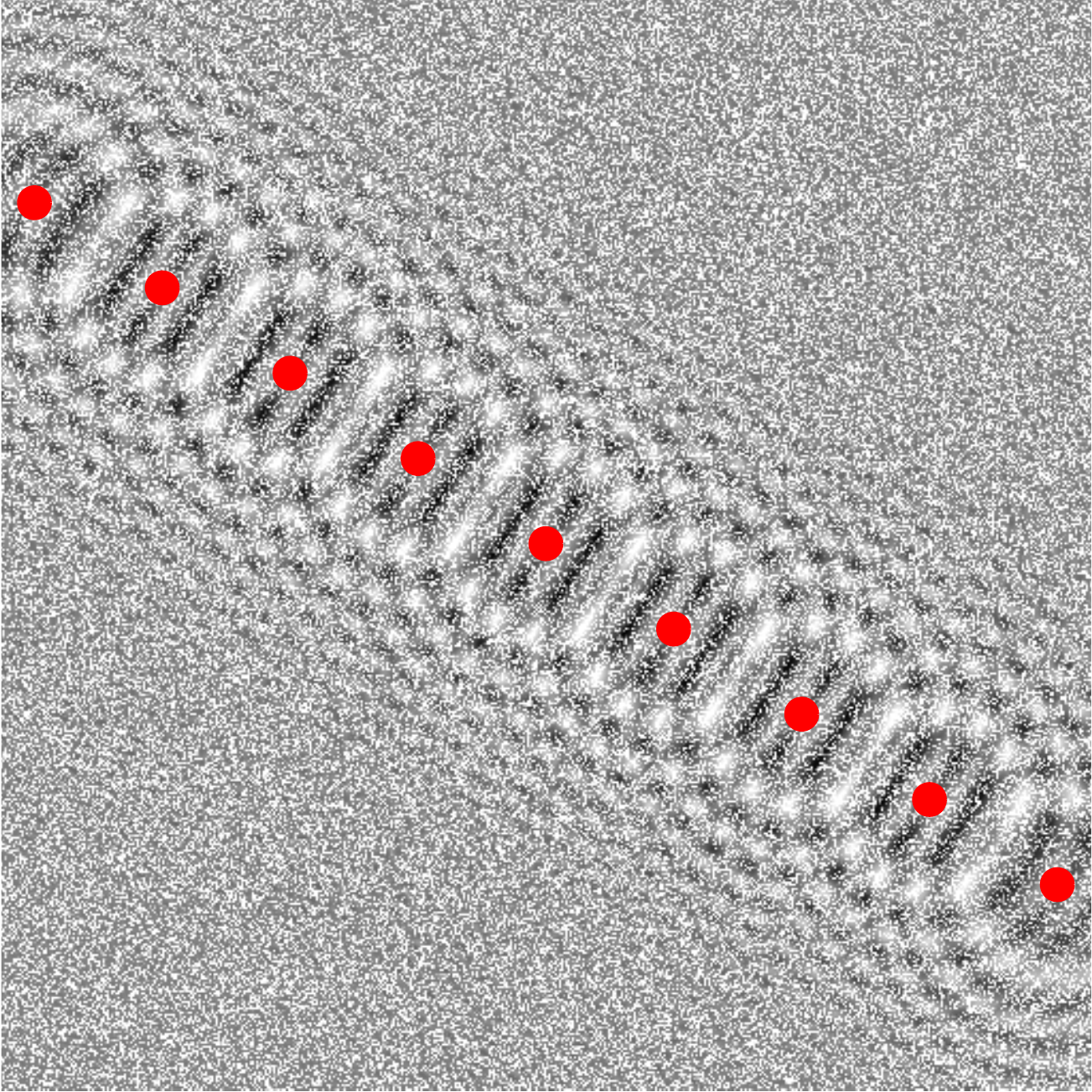}
  \end{tabular}
  \caption{Simulated holograms of spherical particles.
           In the left column the noiseless signals are displayed.
           For reconstruction, the noisy signals
           of the right column are used.
           The dots correspond to the true location of the particles.
           The existence of a suitable regularization parameter
           cannot be shown with condition~(\ref{eq_theorem_fresnel}) 
           of proposition~\ref{theorem_fresnel}, however,
           it is possible to find a regularization parameter so that
           the $\ell^1$-penalized Tikhonov regularization
           recovered all particles exactly.
           The reason why condition~(\ref{eq_theorem_fresnel}) is 
           not fulfilled is that
           in the image on top the particles have
           a too small distance to each other,
           and at the bottom the image was manipulated 
           with unrealistically huge noise.}
  \label{bild_holo2}
\end{figure}

\section{Conclusion}\label{sec-ell1_conclusion}
With the papers~\cite{grasmair2008sparseregularization} and~\cite{Grasmair_suf_nec_ell1},
the analysis of a priori parameter rules for $\ell^1$-penalized Tikhonov
functionals seemed completed. On the common parameter rule $\alpha \asymp \varepsilon$,
linear, i.e.~best possible, convergence is guaranteed.
In this paper we have gone beyond this question by presenting
a parameter rule which ensures exact recovery of 
the unknown support of $u\sparse\in\ell^0$.
Moreover, on that condition we achieve a linear convergence rate
measured in the $\ell^1$ norm,
that comes with a-priori checkable error constants 
which are easier to handle than the ones from~\cite{grasmair2008sparseregularization}.
A side product of our analysis is the proof of convergence 
in $\ell^0$ in the topology of sparse convergence.

Section~\ref{sec:relat-betw-recov} analyzes some implications between
different condition for exact recovery. However, in most cases it
remains open whether the reverse implications also hold and we
postpone this investigation to future work.

Granted, to apply the Neumann $\varepsilon$ERC~(\ref{eq_l1-00-NeumannERC_noise2})
and the Neumann parameter rule~(\ref{eq_l1-00-Neumannparameter_rule2})
one has to know the support $I$.
However, with a certain prior knowledge the correlations
\[
  \CI := \sup\limits_{i\in I} \sum\limits_{\sumstack{j\in I}{j \neq i}} |\ip{K e_i}{K e_j}|
  \quad \text{and} \quad
  \CIC := \sup\limits_{i\in I^\complement} \sum\limits_{j\in I} |\ip{K e_i}{K e_j}|,
\]
can be estimated from above a priori,
especially when the support $I$ is not known exactly.
That way it is possible to obtain a priori computable conditions
for exact recovery.
In section \ref{sec_applications} it has be done
exemplarily for 
characteristic functions convolved with a Fresnel function.
This shows the practical relevance of the condition.

\ack{The authors acknowledge stimulating discussions with Martin
  Benning, Martin Burger and Massimo Fornasier. Moreover, we thank
  R\'emi Gribonval for providing remark~\ref{rem_ERC_NSP}.
  
  Dirk A.~Lorenz is supported by the DFG under grant LO 1436/2-1
  (project ``Sparsity and Compressed Sensing in Inverse Problems'')
  within the Priority Program SPP 1324 ``Extraction of quantitative
  information in complex systems'' and grant LO 1436/3-1 (project
  ``Sparse Exact and Approximate Recovery''). Dennis Trede is
  supported by the BMBF project INVERS under grant 03MAPAH7.  }
 
\section*{References}
 
\bibliographystyle{plain}
\bibliography{literatur}

\end{document}